\documentclass[11pt]{amsart}
\usepackage{amscd,amssymb}
\usepackage{enumerate}
\usepackage[all]{xypic}
\usepackage{multirow}
\usepackage{hhline}
\usepackage{verbatim}
\usepackage{mathdots}
\usepackage{comment}
\usepackage{tikz}
\usepackage{appendix}
\usepackage{stmaryrd}

\usetikzlibrary{decorations.markings}

\allowdisplaybreaks

\theoremstyle{plain}
\newtheorem{Thm}[equation]{Theorem}
\newtheorem*{Thm*}{Theorem}

\newtheorem{Prop}[equation]{Proposition}
\newtheorem{Lem}[equation]{Lemma}

\theoremstyle{definition}

\numberwithin{equation}{section}

\newcommand{\C}{\mathbb C}

\newcommand{\Z}{\mathbb{Z}}
\newcommand{\R}{\mathbb{R}}

\newcommand{\la}{\lambda}
\newcommand{\ep}{\varepsilon}

\newcommand{\Ind}{\operatorname{Ind}}

\newcommand{\sign}{\operatorname{sign}}
\newcommand{\Irr}{\operatorname{Irr}}

\renewcommand{\Re}{\operatorname{Re}}

\newcommand{\GL}{\operatorname{GL}}

\newcommand{\qand}{\quad\text{and}\quad}

\title[A Local converse theorem for Archimedean $\GL(n)$]{A Local converse theorem
  for Archimedean $\GL(n)$}

\author{Moshe Adrian}
\address{Department of Mathematics
Queens College, CUNY
65-30 Kissena Blvd., Queens, NY 11367-15971}
\email{moshe.adrian@qc.cuny.edu}

\author{Shuichiro Takeda}
\address{Mathematics Department, University of Missouri-Columbia, 202
  Math Sciences Building, Columbia, MO, 65211}
\email{takedas@missouri.edu}

\begin{document}

\maketitle

\begin{abstract}
We prove a local
converse theorem for $\GL_n$ over the archimedean local fields which
characterizes an infinitesimal equivalence class of irreducible
admissible representations of $\GL_n(\R)$ or $\GL_n(\C)$
in terms of twisted local gamma factors.
\end{abstract}


\section{Introduction}

Let $F$ be a local field of characteristic 0, and let $\Irr_n$ be the set
of (infinitesimal) equivalence classes of irreducible admissible
representations of $\GL_n(F)$. A so-called local converse theorem for
$\GL_n(F)$ characterizes the set $\Irr_n$ in terms of local factors
with some suitable twists. If $F$ is
non-archimedean, the first major result is the one by Henniart
(\cite{Henniart1}) in which he shows that if two {\it generic}
representations $\pi, \pi'\in\Irr_n$ satisfy
\[
\gamma(s, \pi\otimes\tau, \psi)=\gamma(s, \pi'\otimes\tau, \psi)
\]
for all $\tau\in\Irr_t$ for all $t=1,\dots,n-1$, where the
$\gamma$-factor is the one defined by Jacquet, Piatetski-Shapiro and
Shalika, then
$\pi=\pi'$. Later, Chen (\cite{Chen}) improved this result by requiring
$t$ be only up to $n-2$ with the extra assumption that $\pi$ and
$\pi'$ have the same central character. It had been conjectured by
Jacquet for some time that one only needs $t\leq[\frac{n}{2}]$. And recently this conjecture has been proven by Chai (\cite{Chai}), and Jacquet and Liu
(\cite{JL}) (see also \cite{ALSX, JNS}). Let us also mention that Nien (\cite{Nien}) has shown
an analogous result when $F$ is a finite field.

In this paper, we prove the archimedean analogue of the local converse theorem as follows.
\begin{Thm*}
Let $F=\C$ or $\R$. If  $\pi, \pi'\in\Irr_n$ are generic representations of
$\GL_n(F)$ that satisfy
\[
\gamma(s, \pi\times\chi, \psi)=\gamma(s, \pi'\times\chi, \psi)
\]
for all unitary characters $\chi$ on $F^\times$, then $\pi=\pi'$.
\end{Thm*}

Here, the gamma factors are defined on the ``Galois side" via the local Langlands correspondence (LLC); namely they are the gamma factors of Artin type. The basic idea of our proof is that we pass to the Galois side via the LLC so that the gamma factors, which are essentially products of gamma functions, can be explicitly computed in terms of the data for the corresponding representations of the Weil group.  Then, we will compare poles of the gamma functions. In this sense, what we actually prove is the following assertion: given two local Langlands parameters $\varphi, \varphi':W_F\to\GL_n(\C)$ of generic type, if $\gamma(s, \varphi\otimes\chi, \psi)=\gamma(s, \varphi'\otimes\chi, \psi)$ for all $1$-dimensional characters $\chi$, then $\varphi=\varphi'$. 

We note that it can be shown that the LLC for archimedean $\GL(n)$ is uniquely characterized by the local $L$-factors. This uniqueness result was originally announced by Henniart in \cite[1.10, p.592]{Henniart2}, although his proof has never appeared. Yet, in \cite{Adrian-Takeda} we have proven a refined version of the theorem announced by Henniart. This result will appear elsewhere.

It should be also noted that these gamma factors of Artin type are known to coincide with the local coefficients of Shahidi (\cite{Shahidi}). Moreover, in \cite{Jacquet}, Jacquet has shown that these gamma factors appear as constants of the functional equations satisfied by Rankin-Selberg integrals. The authors do not know if it is possible to prove the local converse theorem as above by using purely representation theoretic methods using this theory of Jacquet without passing to the Galois side, and this is certainly an interesting question to be answered.

\quad

\begin{center}{\bf Notations}\end{center}

Throughout, $F$ is either $\R$ or $\C$. We let $\Irr_n$ be the
set of infinitesimal equivalence classes of irreducible admissible
representations of $\GL_n(F)$. For $z\in F$, we let
$|z|=\sqrt{z\bar{z}}$, so that if $F=\R$, it is the absolute value of
$z$, and if $F=\C$, it is the usual modulus of $z$. We also let
$\|z\|=z\bar{z}=|z|^2$. By a character, we always mean a
quasi-character, and $\Irr_1$ is
the set of characters of $F^\times$. We let $\psi_F$ be the standard
choice of additive character on $F$; namely if $F=\R$, then
$\psi_\R(r)=e^{2\pi i r}$, and if $F=\C$, then
$\psi_\C(z)=\psi_\R\circ\operatorname{Tr}_{\C/\R}(z)=e^{2\pi
  i(z+\bar{z})}$. We let $\Gamma(s)$ be the gamma function.
Recall that $\Gamma(s)$ has no zeroes, and has infinitely many
poles, which are precisely at $s=0, -1,
-2, \dots$, all of which are simple.

Finally, if $w, z \in \mathbb{C}$, then we write
$w \preceq z$ if $z-w \in \mathbb{Z}^{\geq 0}$ . This is a partial order on $\C$. Also $w\prec z$ means $w\preceq z$ and $w\neq z$. For fixed $z, w\in\C$, the gamma functions $\Gamma(s+z)$ and $\Gamma(s+w)$ have a common pole if and only if $z$ and $w$ are comparable under $\preceq$, namely $z-w\in\Z$. We use this fact repeatedly throughout the paper.

\quad

\begin{center}{\bf Acknowledgements}\end{center}

We thank Herv\'{e} Jacquet for suggesting this problem. We also thank the anonymous referee for his/her valuable comments.



\section{Complex case}


In this section, we consider the complex case, so we set $F=\C$.

\subsection{Weil group and its representations}
We let $W_F$ be the Weil group of $F$, namely
\[
W_\C=\C^\times.
\]
Each (not necessarily unitary) character of $\C^\times$, which we also view as a 1-dimensional representation of $W_\C$, is of the form
\[
\chi_{-N, t}(z):=z^{-N}\|z\|^t
\]
for $z\in\C^\times$, where $N\in\Z$ and $t\in\C$.  Let us note that if we write $z=re^{i\theta}$ with $r,
\theta\in\R$ as usual, we have
\[
\chi_{-N,t}(z)=r^{2t-N}e^{-iN\theta}.
\]
But when dealing with the local factors, it seems to be more
convenient to denote each character as $z^{-N}\|z\|^t$ instead of
using $re^{i\theta}$, and hence we choose this convention.
Let us note that
\[
\overline{\chi_{-N,t}}=\chi_{N, t-N},
\]
where $\overline{\chi_{-N,t}}(z):=\overline{\chi_{-N,t}(z)}=\chi_{-N,t}(\bar{z})$ as usual.

Since $W_{\C}$ is abelian, $\chi_{-N, t}$ is the only irreducible semisimple representation of $W_{\C}$, and hence each $n$-dimensional semisimple representation
\[
\varphi:W_\C\rightarrow\GL_n(\C)
\]
is of the form
\begin{equation}\label{E:general_parameter_C}
\varphi=\chi_{-N_1,t_1}\oplus\cdots\oplus\chi_{-N_n, t_n}.
\end{equation}
Note that the contragredient $\varphi^\vee$ is
\[
\varphi^\vee=\chi_{N_1,-t_1}\oplus\cdots\oplus\chi_{N_n, -t_n},
\]
because $\chi_{-N, t}^\vee=\chi_{-N, t}^{-1}=\chi_{N, -t}$.

\subsection{Local factors}
Recall that the $L$-, $\epsilon$- and $\gamma$-factors of the character $\chi_{-N,t}$ are defined as follows:
\begin{align}
\label{E:L-factor-C} L(\chi_{-N, t})&=2(2\pi)^{-(t-\frac{N}{2}+\frac{|N|}{2})}\Gamma(t-\frac{N}{2}+\frac{|N|}{2});\\
\label{E:ep-factor-C}\epsilon(\chi_{-N,t},\psi_\C)&=i^{|N|};\\
\notag\gamma(\chi_{-N,t},\psi_\C)&=\epsilon(\chi_{-N,t},\psi_\C)\frac{L(\chi_{-N, t}^\vee\|\cdot\|)}{L(\chi_{-N, t})}\\
\label{E:gamma-factor-C}&=i^{|N|}(2\pi)^{-1+2t-N}\frac{\Gamma(1-t+\frac{N}{2}+\frac{|N|}{2})}{\Gamma(t-\frac{N}{2}+\frac{|N|}{2})}.
\end{align}
If $\varphi:W_\C\rightarrow\GL_n(\C)$ is an
$n$-dimensional representation as in \eqref{E:general_parameter_C}, we
define the local factors multiplicatively as follows:
\begin{align*}
L(\varphi)&=L(\chi_{-N_1,t_1})\cdots L(\chi_{-N_n, t_n});\\
\epsilon(\varphi,\psi_\C)&=\epsilon(\chi_{-N_1,t_1}, \psi_\C)\cdots\epsilon(\chi_{-N_n, t_n}, \psi_\C);\\
\gamma(\varphi,\psi_\C)&=\gamma(\chi_{-N_1,t_1}, \psi_\C)\cdots\gamma(\chi_{-N_n, t_n}, \psi_\C).
\end{align*}
Note that we have
\[
\gamma(\varphi, \psi_{\C})=\epsilon(\varphi, \psi_{\C})\frac{L(\varphi^\vee\|\cdot\|)}{L(\varphi)}.
\]

\subsection{$\GL(1)$-twist}
Let $\chi_{-M,s}$ be another character on $\C^\times$, and let $\varphi$ be an $n$-dimensional representation of
$W_\C$ as in \eqref{E:general_parameter_C}. Then the twist $\varphi\otimes\chi_{-M,s}$ by  $\chi_{-M,s}$ is given by
\begin{equation}\label{E:twisted_general_parameter_C}
\varphi\otimes\chi_{-M,s}=\chi_{-(N_1+M), t_1+s}\oplus\cdots\oplus\chi_{-(N_n+M), t_n+s}.
\end{equation}
We set
\begin{align*}
L(s, \varphi)&=L(\varphi\otimes\chi_{0, s});\\
\epsilon(s, \varphi,\psi_\C)&=\epsilon(\varphi\otimes\chi_{0, s}, \psi_\C);\\
\gamma(s, \varphi,\psi_\C)&=\gamma(\varphi\otimes\chi_{0, s}, \psi_\C).
\end{align*}
We then have
\[
\gamma(s, \varphi,\psi_{\C})=\epsilon(\varphi, \psi_{\C})\frac{L(1-s, \varphi^\vee)}{L(s, \varphi)}.
\]

\subsection{Local Langlands correspondence for $\GL_n(\C)$}
By the archimedean local Langlands correspondence, originally established by Langlands (\cite{Langlands}), there is a one-to-one
correspondence between the set $\Irr_n$ of (infinitesimal equivalence classes) of irreducible admissible representations of $\GL_n(\C)$ and the set $\Phi_n$ of (conjugacy classes of) all continuous semisimple $n$-dimensional representations of $W_{\C}$. This correspondence can be fairly explicitly described as follows. For each
\[
\varphi=\chi_{-N_1,t_1}\oplus\cdots\oplus\chi_{-N_n, t_n}\in\Phi_n,
\]
consider the (normalized) induced representation
\[
I(\varphi):=\Ind_{B(\C)}^{\GL_n(\C)}\chi_{-N_1,t_1}\otimes\cdots\otimes\chi_{-N_n, t_n},
\]
where $B(\C)$ is the Borel subgroup of $\GL_n(\C)$ and the character $\chi_{-N_1,t_1}\otimes\cdots\otimes\chi_{-N_n, t_n}$ is viewed as a character on $B(\C)$ as usual. Let us reorder the constituents of $\varphi$ in the Langlands situation, which means
\[
\Re(t_1)\geq \cdots\geq\Re(t_n).
\]
By the Langlands quotient theorem, $I(\varphi)$ has a unique irreducible quotient, which we denote by $\pi_\varphi$. Then the local Langlands correspondence is obtained by the map
\[
\Phi_n\longrightarrow\Irr_n,\quad \varphi\mapsto\pi_{\varphi}.
\]

\subsubsection{Genericity conditions}
It is well-known that the (full) induced representation $I(\varphi)$ is always generic. (See, for example, \cite[Theorem 15.4.1, p.381]{Wallach}.) The following proposition characterizes when the Langlands quotient $\pi_\varphi$ is generic.
\begin{Prop}\label{P:genericity_condition_complex}
Let
\[
\varphi=\chi_{-N_1,t_1}\oplus\cdots\oplus\chi_{-N_n, t_n}
\]
be such that
\[
N_1\leq \cdots\leq N_n.
\]
Then the following are all equivalent.
\begin{enumerate}[(1)]
\item The representation $\pi_\varphi$ that corresponds to $\varphi$ under the local Langlands correspondence is generic.
\item $\pi_\varphi=I(\varphi)$, namely $I(\varphi)$ is irreducible.
\item For all $i\leq j$, whenever $t_j-t_i\in\Z$, we have
\[
0\leq t_j-t_i\leq N_j-N_i.
\]
In particular, if $t_j-t_i\in\Z$ then $t_i\preceq t_j$, where we recall from the notation section that $t_i\preceq t_j$ means $t_j-t_i\in\Z^{\geq 0}$. (Note that if $t_j-t_i\notin \Z$ then there is no condition.)
\item The Rankin-Selberg $L$-factor
\[
L(s, \varphi\otimes\varphi^\vee):=L(\varphi\otimes\varphi^\vee\otimes\chi_{0, s})
\]
is holomorphic at $s=1$.
\end{enumerate}
\end{Prop}
\begin{proof}
The equivalence of (1) and (2) is well-known. The equivalence of (2) and (3) is a special case of \cite{Speh-Vogan}, though, presumably, the case of $\GL_n(\C)$ had been known much before. Since the authors were not able to find an explicit reference for $\GL_n(\C)$, we reproduce essential parts of the proof.

First consider the principal series
\[
\Ind_{B(\C)}^{\GL_n(\C)}\chi_1\otimes\cdots\otimes\chi_n,
\]
where $\chi_i:\C^\times\to\C^\times$ is a character. This is reducible if and only if for some $i\neq j$ the character $\chi_i\chi_j^{-1}$ is of the form
\[
\chi_i\chi_j^{-1}(z)=z^p\overline{z}^q\qquad (p-q\in\Z),
\]
where either both $p$ and $q$ are in $\Z^{>0}$ or both $p$ and $q$ are in $\Z^{<0}$. (One can prove this by reducing to the $\GL_2(\C)$ situation by induction in stages and applying \cite[Theorem 6.2]{Jacquet-Langlands} or one may apply the general result of \cite[Theorem 1.1]{Speh-Vogan} to $\GL_n(\C)$.)

Now for each $i<j$ we have
\[
(\chi_{-N_i, t_i})(\chi_{-N_j, t_j})^{-1}(z)=z^{-N_i+N_j}\|z\|^{t_i-t_j}=z^{N_j-N_i+t_i-t_j}\overline{z}^{t_i-t_j}.
\]
Noting $N_i\leq N_j$, we know that $I(\varphi)$ is reducible if and only if $t_i-t_j\in\Z^{>0}$ or otherwise both $N_j-N_i+t_i-t_j\in\Z^{<0}$ and $t_i-t_j\in\Z^{<0}$. Hence $I(\varphi)$ is irreducible if and only if whenever $t_i-t_j\in\Z\smallsetminus\{0\}$ we have $t_i-t_j\notin\Z^{>0}$ and $N_j-N_i+t_i-t_j\notin\Z^{<0}$. One can then see that these conditions are precisely (3).

We show the equivalence of (3) and (4). Since
\[
\varphi\otimes\varphi^\vee\otimes\chi_{0, s}=\sum_{i, j}\chi_{-N_i, t_i}\otimes\chi_{N_j,-t_j}\otimes\chi_{0, s}=\sum_{i, j}\chi_{-(N_i-N_j),\, s+t_i-t_j},
\]
we have
\begin{align*}
L(s, \varphi\otimes\varphi^\vee)&=\prod_{i, j}L(\chi_{-(N_i-N_j),\, s+t_i-t_j})\\
&=F(s)\prod_i\Gamma(s)\prod_{i<j}\Gamma(s+t_i-t_j-N_i+N_j)\Gamma(s+t_j-t_i),
\end{align*}
where $F(s)$ is a holomorphic function without zeros. (Here to compute the $L$-factors we used that $N_i$'s are in the increasing order.) Hence $L(s, \varphi\otimes\varphi^\vee)$ is holomorphic at $s=1$ if and only if
\[
t_i-t_j-N_i+N_j\notin\Z^{<0}\quad\text{and}\quad t_j-t_i\notin\Z^{<0}.
\]
But this condition is equivalent to
\[
0\leq t_j-t_i\leq N_j-N_i
\]
whenever $t_j-t_i\in\Z$ with $i\leq j$.
\end{proof}

\subsection{Local Converse Theorem for $\GL_n(\C)$}

For two characters $\chi_{N, t}$ and $\chi_{M, s}$, we define
\[
\chi_{N, t}\sim\chi_{M, s}\quad \text{if $t-s\in\Z$}.
\]
This is certainly an equivalence relation. Then, given a Langlands parameter $\varphi$ of $\GL_n(\C)$, by grouping the constituents by this equivalence relation we can write
\[
\varphi=\varphi_1\oplus\cdots\oplus\varphi_k,
\]
where all the constituents of $\varphi_i$ are equivalent under $\sim$ and the constituents of different $\varphi_i$ and $\varphi_j$ are inequivalent under $\sim$. Then we know that in the $\gamma$-factor
\[
\gamma(s, \varphi, \psi_{\C})=\gamma(s, \varphi_1, \psi_{\C})\cdots\gamma(s, \varphi_n, \psi_{\C})
\]
the zeros and the poles coming from $\gamma(s, \varphi_i, \psi_{\C})$ do not interfere with those coming from $\gamma(s, \varphi_j, \psi_{\C})$ for $j\neq i$.

Let us first prove the following.
\begin{Prop}
Let
\[
\varphi=\chi_{-N_1, t_1}\oplus\cdots\oplus\chi_{-N_n, t_n}\quad\text{and}\quad\varphi'=\chi_{-N'_1, t'_1}\oplus\cdots\oplus\chi_{-N'_{n'}, t'_{n'}}
\]
be generic parameters of $\GL_n(\C)$ and $\GL_{n'}(\C)$, respectively, such that all the constituents $\chi_{-N_i, t_i}$'s and $\chi_{-N'_j, t'_j}$'s are equivalent under $\sim$, namely $t_i-t_j'\in\Z$ for all $i, j$. Assume
\begin{equation}\label{E:idenity_gamma_factor_complex}
F_{\chi}(s)\gamma(s, \varphi\otimes\chi, \psi_{\C})=\gamma(s, \varphi'\otimes\chi, \psi_{\C})
\end{equation}
for all characters $\chi$, where $F_{\chi}(s)$ is a meromorphic function (depending on $\chi$) whose poles and zeros do not interfere with those from the gamma factors. Then $\varphi=\varphi'$ (and hence $n=n'$).
\end{Prop}
\begin{proof}
Since all the constituents $\chi_{-N_i, t_i}$'s and $\chi_{-N'_j, t'_j}$'s are equivalent under $\sim$, there exists $s_0$ with $\Re(s_0)$ large enough such that all of the $\gamma(s, \chi_{-N_i, t_i}, \psi_{\C})$ and $\gamma(s, \chi_{-N'_j, t'_j}, \psi_{\C})$ have a simple pole at $s=s_0$, so that, at $s=s_0$, $\gamma(s, \varphi\otimes\chi, \psi_{\C})$ has a pole of order $n$ and $\gamma(s, \varphi'\otimes\chi, \psi_{\C})$ has a pole of order $n'$. Hence we have $n=n'$.

Without loss of generality, we may assume
\[
N_1\leq\cdots\leq N_n\quad\text{and}\quad N'_1\leq\cdots\leq N'_n.
\]
Since $\varphi$ is generic and $t_j-t_i\in\Z$, by Proposition \ref{P:genericity_condition_complex} (3) we have
\begin{equation}\label{E:genericity_condition_complex_proof}
0\leq t_j-t_i\leq N_j-N_i
\end{equation}
for $i\leq j$. In particular, $t_1\preceq\cdots\preceq t_n$. And similarly for $t_i'$'s and $N_i'$'s.

Let $\chi=\chi_{-M, 0}$ be such that $M+N_i>0$ and $M+N'_i>0$ for all $i$. Then (the reciprocal of) the identity \eqref{E:idenity_gamma_factor_complex} is equivalent to
\begin{equation}\label{E:gamma_fanction_identity_complex}
F(s)\prod_{i=1}^n\frac{\Gamma(s+t_i)}{\Gamma(1-s-t_i+N_i+M)}=\prod_{i=1}^n\frac{\Gamma(s+t'_i)}{\Gamma(1-s-t'_i+N'_i+M)},
\end{equation}
where $F(s)$ is a meromorphic function whose poles and zeros do not interfere.

Set $M$ to be large enough so that all of the gamma functions in the denominators in \eqref{E:gamma_fanction_identity_complex} are holomorphic at $s=-t_1, \dots, -t_n$. Then on the left-hand side we have a pole at $s=-t_1$ coming from $\Gamma(s+t_1)$. Hence on the right-hand side we must have a pole at $s=-t_1$ from some $\Gamma(s+t'_i)$. If $t'_1$ is such that $t_1'\succ t_1$ (strict inequality), then since $t_i'$'s are in the increasing order (with respect to $\preceq$) we never have a pole at $s=-t_1$ for any of the $\Gamma(s+t'_i)$'s. Hence $t_1'\preceq t_1$. By switching the roles of $t_1$ and $t_1'$ we have $t_1'\succeq t_1$. Hence we have $t_1=t_1'$. Thus, $\Gamma(s+t_1)$ and $\Gamma(s+t_1')$ can be removed from \eqref{E:gamma_fanction_identity_complex}. Arguing inductively we have
\[
t_i=t_i'
\]
for all $i=1,\dots, n$.

Thus we can reduce \eqref{E:gamma_fanction_identity_complex} to
\begin{equation}\label{E:gamma_fanction_identity_complex2}
F(s)\prod_{i=1}^n\Gamma(1-s-t'_i+N'_i+M)=\prod_{i=1}^n\Gamma(1-s-t_i+N_i+M),
\end{equation}
where $M$ is a fixed integer. Let $k, \ell\in\{1,\dots,n\}$ be such that
\[
t_k-N_k\succeq t_i-N_i\quad\text{and}\quad t_\ell'-N_\ell'\succeq t_i'-N_i'
\]
for all $i$. Then at $s=-t_k+N_k+M$ the right-hand side (and hence the left-hand side) is holomorphic, which implies $t_k-N_k\preceq t_\ell'-N_\ell'$. By switching the roles we obtain $t_k-N_k\succeq t_\ell'-N_\ell'$ and hence $t_k-N_k= t_\ell'-N_\ell'$. By arguing inductively we obtain
\[
\{t_1-N_1,\dots, t_q-N_q\}=\{t_1'-N_1',\dots,t_q'-N_q'\}
\]
as multisets.

Now, we will show $N_i=N_i'$ for all $i=1,\dots, n$. By the above identity of the multisets we must have $t_1-N_1=t_i'-N_i'$ for some $i$. Since we already know $t_i=t_i'$, we have
\[
N_i'-N_1=t_i'-t_1=t_i'-t_1'\leq N_i'-N_1',
\]
where the last inequality is by the genericity condition \eqref{E:genericity_condition_complex_proof}. Hence we have $N_1'\leq N_1$. Also we must have $t_1'-N_1'=t_j-N_j$ for some $j$. By applying the same argument, we must have $N_1\leq N_1'$. Thus we must have $N_1=N_1'$. By arguing inductively we have
\[
N_i=N_i'
\]
for all $i=1,\dots, n$.
\end{proof}

Now, we are ready to prove the local converse theorem.
\begin{Thm}
Let $\pi$ and $\pi'$ be generic irreducible admissible representations of $\GL_n(\C)$. Assume that
\[
\gamma(s, \pi\times\chi, \psi_{\C})=\gamma(s, \pi'\times\chi, \psi_{\C})
\]
for all unitary characters $\chi$. Then $\pi=\pi'$.
\end{Thm}
\begin{proof}
Let us first note that if $\chi$ is not unitary, then $\gamma(s, \pi\times\chi, \psi_{\C})=\gamma(s+t, \pi\times\chi', \psi_{\C})$ for some $t\in\C$ and some unitary character $\chi'$. Hence we may assume that the identity of the gamma factors holds for all (not necessarily unitary) characters $\chi$.

Let $\varphi$ and $\varphi'$ be the Langlands parameters of $\GL_n(\C)$ corresponding to $\pi$ and $\pi'$, respectively. Let us write
\[
\varphi=\varphi_1\oplus\cdots\oplus\varphi_k\quad\text{and}\quad\varphi'=\varphi_1'\oplus\cdots\oplus\varphi'_{k'},
\]
where all the constituents of each $\varphi_j$ are equivalent under $\sim$ and the constituents of $\varphi_i$ and $\varphi_j$ are inequivalent under $\sim$ for $i\neq j$, and similarly for $\varphi'$.

We then have
\[
\prod_{i=1}^k\gamma(s, \varphi_i\otimes\chi, \psi_{\C})=\prod_{i=1}^{k'}\gamma(s, \varphi_i'\otimes\chi, \psi_{\C}).
\]
Note that, for $i\neq j$, the gamma factors $\gamma(s, \varphi_i\otimes\chi, \psi_{\C})$ and $\gamma(s, \varphi_j\otimes\chi, \psi_{\C})$ do not share a zero or a pole, and similarly for $\varphi'$.

Assume that $\varphi$ and $\varphi'$ do not share any constituents equivalent under $\sim$. Then $\gamma(s, \varphi\otimes\chi, \psi_{\C})$ and $\gamma(s, \varphi'\otimes\chi, \psi_{\C})$ do not share a zero or pole. So there are at least some $\varphi_i$ and $\varphi_j'$ having constituents equivalent under $\sim$. By reordering the indices, we may assume $i=j=1$. Then the equality of the gamma factors is written as
\[
F_{\chi}(s)\gamma(s,\varphi_1\otimes\chi, \psi_{\C})=\gamma(s, \varphi_1'\otimes\chi, \psi_{\C}),
\]
where $F_{\chi}(s)$ is a meromorphic function whose poles and zeros do not interfere with those of the above two gamma factors. Hence by the above proposition we have $\varphi_1=\varphi_1'$. Arguing inductively, we conclude $\varphi=\varphi'$.
\end{proof}


\section{Real Case}


In this section, we consider the real case, so we set $F=\R$.

\subsection{Weil group and its representations}
Recall that the Weil group $W_{\R}$ of $\R$ is defined as
\[
W_\R=\C^\times\cup j\C^\times,\quad j^2=-1,\quad jzj^{-1}=\bar{z},
\]
where $z\in\C^\times$. We naturally view $W_\C=\C^\times$ as a subgroup of
$W_\R$. Note that $\R^\times\cong W_{\R}^{ab}$ because we have a surjective map
\begin{equation}\label{E:W_R}
W_\R\longrightarrow \R^\times, \quad z\mapsto z\bar{z},\;\; j\mapsto -1,
\end{equation}
whose kernel is the commutator group $[W_\R, W_\R]$, which is of the form
$\{z\in\C^\times:|z|=1\}$.

An irreducible representation of $W_\R$ is 1 or 2 dimensional. If it
is 1-dimensional, it factors through $W_\R^{ab}\cong\R^\times$ and
hence is identified with a character, which is of the form
\[
\la_{\ep, t}(r):=r^{-\ep}|r|^t=\sign(r)^{\ep}|r|^{t-\ep},\quad r\in\R^\times,
\]
where $\ep\in\{0,1\}$, $t\in\C$ and $\sign$ is the sign character. Also we often write $\la_{0, t}=|\cdot|^t$. If it
is 2-dimensional, it is of the form
\[
\varphi_{-N,t}:=\Ind_{W_\C}^{W_\R}\chi_{-N,t},
\]
where $\chi_{-N, t}$ is the character on $\C^\times$ as before, namely
\[
\chi_{-N, t}(z)=z^{-N}\|z\|^t
\]
for $z\in\C^\times$.

If $N=0$ then the representation $\varphi_{-N,t}$ is not irreducible but
we have
\[
\varphi_{0, t}=\la_{0, t}\oplus\la_{1, t+1}.
\]
But otherwise it is irreducible. Furthermore, since
$\Ind_{W_\C}^{W_\R}\chi_{-N,t}=\Ind_{W_\C}^{W_\R}\overline{\chi_{-N,t}}$,
we have
\begin{equation*}
\varphi_{-N,t}=\varphi_{N, t-N}.
\end{equation*}
Hence we may and do assume that $N\geq 0$. Also we consider $\la_{0, t}\oplus\la_{1, t+1}$ as the induced representation $\varphi_{0, t}$. In
general, an $n$-dimensional representation
$\varphi:W_\R\rightarrow\GL_n(\C)$ is of the form
\begin{equation}\label{E:general_parameter_R}
\varphi=\left(\la_{\ep_1,t_1}\oplus\cdots\oplus\la_{\ep_p,t_p}\right)\oplus
\left(\varphi_{-N_1,u_1}\oplus\cdots\oplus\varphi_{-N_q, u_q}\right)
\end{equation}
where we may assume that $N_i\geq 0$ for all $i$ and a representation of the form $\la_{0, t}\oplus\la_{1, t+1}$ is treated as $\varphi_{0, t}$.

Note that
\[
{\la_{\ep, t}}^\vee={\la_{\ep, t}}^{-1}=\la_{\ep, -t+2\ep}
\]
and
\[
{\varphi_{-N, t}}^\vee=\Ind_{W_{\C}}^{W_{\R}}\overline{\chi_{-N, t}}^{-1}=\varphi_{-N, N-t}.
\]

\subsection{$L$-, $\epsilon$- and $\gamma$-factors}

For the 1-dimensional $\la_{\ep, t}$, the $L$-, $\epsilon$- and $\gamma$-factors are defined as follows.
\begin{align*}
L(\la_{\ep,t})&=\pi^{-\frac{t}{2}}\Gamma\left(\frac{t}{2}\right);\\
\epsilon(\la_{\ep,t},\psi_\R)&=(-i)^\ep;\\
\gamma(\la_{\ep,t},\psi_\R)&=\epsilon(\la_{\ep,t},\psi_\R)\frac{L({\la_{\ep,t}}^\vee|\cdot|)}{L(\la_{\ep,t})}\\
&=(-i)^{\ep}\pi^{t-\ep-\frac{1}{2}}\frac{\Gamma\left(\frac{1-t+2\ep}{2}\right)}{\Gamma\left(\frac{t}{2}\right)}.
\end{align*}

For the 2-dimensional representation $\varphi_{-N,t}$ with $N\geq 0$, the local factors are defined as follows.
\begin{align*}
L(\varphi_{-N,t})&=L(\chi_{-N,t})=2(2\pi)^{-t}\Gamma(t);\\
\epsilon(\varphi_{-N,t},\psi_\R)&=-i\cdot\epsilon(\chi_{-N,t},\psi_\C)=-i^{|N|+1};\\
\gamma(\varphi_{-N,t},\psi_{\R})&=\epsilon(\varphi_{-N,t},\psi_\R)\cdot\frac{L({\varphi_{-N,t}}^\vee|\cdot|)}{L(\varphi_{-N,t})}\\
 &=-i^{|N|+1}(2\pi)^{2t-N-1}\cdot\frac{\Gamma(1-t+N)}{\Gamma(t)}.
\end{align*}
In general, if $\varphi:W_\R\rightarrow\GL_n(\C)$ is an
$n$-dimensional representation as in \eqref{E:general_parameter_R}, we
again define the local factors multiplicatively as
\begin{align*}
L(\varphi)&=L(\la_{\ep_1,t_1})\cdots L(\la_{\ep_p,t_p})\cdot
L(\varphi_{-N_1,u_1})\cdots L(\chi_{-N_q, u_q});\\
\epsilon(\varphi,\psi_\R)
&=\epsilon(\la_{\ep_1,t_1}, \psi_\R)\cdots\epsilon(\la_{\ep_p,t_p}, \psi_\R)\cdot
\epsilon(\varphi_{-N_1,u_1}, \psi_\R)\cdots \epsilon(\chi_{-N_q, u_q},\psi_\R);\\
\gamma(\varphi,\psi_\R)
&=\gamma(\la_{\ep_1,t_1}, \psi_\R)\cdots\gamma(\la_{\ep_p,t_p}, \psi_\R)\cdot
\gamma(\varphi_{-N_1,u_1}, \psi_\R)\cdots \gamma(\chi_{-N_q, u_q},\psi_\R).
\end{align*}
Let us note that for the parameter $\varphi_{0, t}$ one can check
\begin{align*}
L(\varphi_{0, t})&=L(\la_{0,t})L(\la_{1,t+1}),\\
\epsilon(\varphi_{0, t},\psi_\R)&=\epsilon(\la_{0,t},\psi_\R)\epsilon(\la_{1,t+1},\psi_\R),\\
\gamma(\varphi_{0, t},\psi_\R)&=\gamma(\la_{0,t},\psi_\R)\gamma(\la_{1,t+1},\psi_\R)
\end{align*}
by using the duplication formula $\Gamma(\frac{t}{2})\Gamma(\frac{t+1}{2})=2^{1-t}\sqrt{\pi}\Gamma(t)$.

%

\subsection{$\GL(1)$-twist}
Let $\la_{\ep,t}$ and $\la_{\delta, s}$ be characters on $\R^\times$. We set
\[
\eta=\begin{cases}2,\quad\text{if $\ep=\delta=1$};\\
0,\quad\text{otherwise},
\end{cases}
\]
so that
\[
\ep+\delta-\eta\in\{0,1\}\qand \ep+\delta-\eta=\ep+\delta\mod{2}.
\]
We then have
\begin{equation}\label{E:twisted_general_parameter_R_by_GL_1}
\la_{\ep, t}\otimes\la_{\delta,s}=
\la_{\ep+\delta-\eta,\, s+t-\eta},
\end{equation}
and hence
\begin{align}
L(\la_{\ep,t}\otimes\la_{\delta,s})&=\pi^{-\frac{s+t-\eta}{2}}\Gamma\left(\frac{s+t-\eta}{2}\right);\notag\\
\epsilon(\la_{\ep,t}\otimes\la_{\delta,s},\psi_\R)&=(-i)^{\ep+\delta-\eta};\notag\\
\label{E:twisted_gamma_factor_real}\gamma(\la_{\ep,t}\otimes\la_{\delta,s},\psi_\R)&=(-i)^{\ep+\delta-\eta}\pi^{s+t-\ep-\delta+\eta-\frac{1}{2}}
\frac{\Gamma\left(\frac{1-s-t+2(\ep+\delta)-\eta}{2}\right)}{\Gamma\left(\frac{s+t-\eta}{2}\right)}.
\end{align}

For the 2-dimensional parameter
$\varphi_{-N,t}=\Ind_{W_\C}^{W_\R}\chi_{-N, t}$, the twisted parameter
$\varphi_{-N,t}\otimes\la_{\delta,s}$ is computed as
\begin{align*}
\varphi_{-N,t}\otimes\la_{\delta,s}
&=\Ind_{W_\C}^{W_\R}(\chi_{-N, t}\otimes(\la_{\delta, s}\circ
N_{\C/\R}))\\
&=\Ind_{W_\C}^{W_\R}(\chi_{-N, t}\otimes\chi_{0, s-\delta})\\
&=\Ind_{W_\C}^{W_\R}\chi_{-N, t+s-\delta},
\end{align*}
and
\begin{equation*}
\varphi_{-N, t}\otimes\la_{\delta, s}=\varphi_{-N, t+s-\delta}.
\end{equation*}
Accordingly, we have
\begin{align*}
L(\varphi_{-N,t}\otimes\la_{\delta, s})&=2(2\pi)^{-(s+t-\delta)}\Gamma(s+t-\delta);\\
\epsilon(\varphi_{-N,t}\otimes\la_{\delta, s},\psi_\R)&=-i^{|N|+1};\\
\gamma(\chi_{-N,t}\otimes\la_{\delta, s},\psi_{\R})&=-i^{|N|+1}(2\pi)^{2(s+t-\delta)-N}\cdot\frac{\Gamma(1-s-t+\delta+N)}{\Gamma(s+t)}.
\end{align*}

If $\varphi:W_\R\rightarrow\GL_n(\C)$ is an $n$-dimensional representation as in \eqref{E:general_parameter_R}, we have
\[
\varphi\otimes\la_{\delta,  s}=
\left(\la_{\ep_1+\delta-\eta_1,s+t_1-\eta_1} \oplus\cdots\oplus\la_{\ep_p+\delta-\eta_p,s+t_p-\eta_p}\right)\oplus
\left(\varphi_{-N_1,s+u_1-\delta}\oplus\cdots\oplus\varphi_{-N_q, s+u_q-\delta}\right),
\]
where $\eta_i$ is defined as before, namely $\eta_i=2$ if $\ep_i=\delta_i=1$ and $\eta_i=0$ otherwise. Accordingly we have
\[
\gamma(\varphi\otimes\la_{\delta, s},\psi_{\R})=F(s)
\prod_{i=1}^p\frac{\Gamma\left(\frac{1-s-u_i+2(\ep_i+\delta)-\eta_i}{2}\right)}{\Gamma\left(\frac{s+u_i-\eta_i}{2}\right)}
\prod_{i=1}^q\frac{\Gamma(1-s-t_i+\delta_i+N_i)}{\Gamma(s+t_i-\delta_i)},
\]
where $F(s)$ is a holomorphic function without a zero.

We set
\begin{align*}
L(s, \varphi)&=L(\varphi\otimes\chi_{0, s});\\
\epsilon(s, \varphi,\psi_\R)&=\epsilon(\varphi\otimes\chi_{0, s}, \psi_\R);\\
\gamma(s, \varphi,\psi_\R)&=\gamma(\varphi\otimes\chi_{0, s}, \psi_\R).
\end{align*}
We then have
\[
\gamma(s, \varphi,\psi_{\R})=\epsilon(\varphi, \psi_{\R})\frac{L(1-s, \varphi^\vee)}{L(s, \varphi)}.
\]

\subsection{$\GL(2)$-twist}
For $2$-dimensional representations $\varphi_{-N,t}$ and $\varphi_{-M,s}$ of $W_\R$, we have
\begin{align*}
\varphi_{-N,t}\otimes\varphi_{-M,s}=&
\left(\Ind_{W_\C}^{W_\R}\chi_{-N,t}\right)\otimes
\left(\Ind_{W_\C}^{W_\R}\chi_{-M,s}\right)\\
=&\left(\Ind_{W_\C}^{W_\R}\chi_{-N,t}\cdot\chi_{-M,s}\right)\oplus
\left(\Ind_{W_\C}^{W_\R}\chi_{-N,t}\cdot\overline{\chi_{-M,s}}\right)\\
=&\left(\Ind_{W_\C}^{W_\R}\chi_{-(N+M),t+s}\right)\oplus
\left(\Ind_{W_\C}^{W_\R}\chi_{-N,t}\cdot\chi_{M,s-M}\right)\\
=&\varphi_{-(N+M), t+s}\oplus
\left(\Ind_{W_\C}^{W_\R}\chi_{-(N-M),t+s-M}\right)\\
=&\varphi_{-(N+M), t+s}\oplus\varphi_{-(N-M), t+s-M}.
\end{align*}

\subsection{Local Langlands correspondence for $\GL_n(\R)$}
By the archimedean local Langlands correspondence, originally established by Langlands (\cite{Langlands}), there is a one-to-one
correspondence between the set $\Irr_n$ of (infinitesimal equivalence classes) of irreducible admissible representations of $\GL_n(\R)$ and the set $\Phi_n$ of (conjugacy classes of) all continuous semisimple $n$-dimensional representations of $W_{\R}$. This correspondence is explicitly described as follows.

The 1-dimensional representation $\lambda_{\ep, t}$ corresponds to the character on $\GL_1(\R)$ in the obvious way. The 2-dimensional representation $\varphi_{-N, t}$ corresponds to the representation of $\GL_2(\R)$ of the form
\[
D_{N}\otimes|\det|^{t-\frac{N}{2}},
\]
where $D_N$ is the discrete series representation of $\GL_2(\R)$ if $N\geq 1$ and the limit of discrete series if $N=0$.

In general, let
\[
\varphi=\varphi_1\oplus\cdots\oplus\varphi_k\in\Phi_n,
\]
where each $\varphi_i$ is either $\la_{\ep_i, t_i}$ or $\varphi_{-N_i, t_i}$ with $N_i\geq 0$, with the proviso that $\lambda_{0, t}\oplus\lambda_{1, t+1}$ is considered as $\varphi_{0, t}$. For each $i$, we let $\pi_i$ be the representation of $\GL_{n_i}(\R)$ corresponding to $\varphi_i$ as above, so that $\pi_i$ is a character with $n_i=1$ or a (limit of) discrete series with $n_i=2$. Note that $n_1+\cdots+n_k=n$. Let $P(\R)$ be the $(n_1,\dots,n_k)$-parabolic of $\GL_n(\R)$, so that the Levi part is $\GL_{n_1}(\R)\times\cdots\times\GL_{n_k}(\R)$, where $n_i=1, 2$. Consider the (normalized) induced representation
\[
I(\varphi):=\Ind_{P(\R)}^{\GL_n(\R)}\pi_1\otimes\cdots\otimes\pi_k.
\]
Let us reorder the constituents of $\varphi$ in the Langlands situation, which means
\[
\Re(t_1)\geq\cdots\geq\Re(t_k).
\]
By the Langlands quotient theorem, the induced representation $I(\varphi)$ has a unique irreducible quotient (the Langlands quotient), which we denote by $\pi_\varphi$. Then the local Langlands correspondence is obtained by the map
\[
\Phi_n\longrightarrow\Irr_n,\quad \varphi\mapsto\pi_{\varphi}.
\]

\subsection{Genericity conditions}
It is well-known that the induced representation $I(\varphi)$ is generic. (See, for example, \cite[Theorem 15.4.1, p.381]{Wallach}.) The following proposition characterizes when the Langlands quotient $\pi_\varphi$ is generic.
\begin{Prop}\label{P:genericity_condition_real}
Let
\[
\varphi=\left(\la_{\ep_1,t_1}\oplus\cdots\oplus\la_{\ep_p,t_p}\right)\oplus
\left(\varphi_{-N_1,u_1}\oplus\cdots\oplus\varphi_{-N_q, u_q}\right)
\]
be a Langlands parameter, where $\la_{0, t}\oplus\la_{1, t+1}$ (if there is any) is considered as $\varphi_{0, t}$. Assume
\[
\Re(t_1)\leq\cdots\leq\Re(t_p)\quad\text{and}\quad N_1\leq\cdots\leq N_q.
\]

Then the following are all equivalent.
\begin{enumerate}[(1)]
\item The representation $\pi_\varphi$ that corresponds to $\varphi$ under the local Langlands correspondence is generic.
\item $\pi_\varphi=I(\varphi)$, namely $I(\varphi)$ is irreducible.
\item All of the following three hold:
\begin{enumerate}[(a)]
\item If $t_i-t_j\in\Z$, then $t_i-t_j\in 2\Z$.
\item If $u_i-t_j\in\Z$, then $-\ep_j\leq u_i-t_j\leq N_i-\ep_j$.
\item If $u_i-u_j\in\Z$, then $0\leq u_j-u_i\leq N_j-N_i$ for $i\leq j$. In particular, if $u_i-u_j\in\Z$ then $u_i\preceq u_j$ for $i\leq j$.
\end{enumerate}
(Note that if $t_i-t_j\notin \Z$, $u_i-t_j\notin\Z$ or $u_i-u_j\notin\Z$, then there is no condition for the corresponding case.)
\item The Rankin-Selberg $L$-factor
\[
L(s, \varphi\otimes\varphi^\vee):=L(\varphi\otimes\varphi^\vee\otimes\chi_{0, s})
\]
is holomorphic at $s=1$.
\end{enumerate}
\end{Prop}
\begin{proof}
The equivalence of (1) and (2) is well-known. The equivalence of (2) and (3) is obtained by Speh in her Ph.D thesis, and the results are nicely summarized in \cite[Theorem 10b, p.164]{Moeglin}. (But one has to translate \cite{Moeglin} to our situation. For doing that, reorder the constituents of $\varphi$ in the Langlands situation, and use that her $p_i$ is our $N_i$, her $s_i$ with $n_i=1$ is our $t_i$, and her $s_i$ with $n_i=2$ is our $u_i-\frac{N_i}{2}$. The details are left to the reader.)

To show the equivalence of (3) and (4), note that, since
\[
\varphi^\vee=\left(\la_{\ep_1,\, -t_1+2\ep_1}\oplus\cdots\oplus\la_{\ep_p,\, -t_p+2\ep_p}\right)\oplus
\left(\varphi_{-N_1, N_1-u_1}\oplus\cdots\oplus\varphi_{-N_q, N_q-u_q}\right),
\]
one can compute
\begin{align*}
&\varphi\otimes\varphi^{\vee}\otimes\lambda_{0,s}\\
&=\bigoplus_{i,j}\lambda_{\ep_i+\ep_j,\, s+t_i-t_j+2\ep_j-\gamma_{ij}}\bigoplus_{i,j}\varphi_{-N_i,\,s+u_i-t_j+\ep_j} \\
 &\quad\bigoplus_{i,j}\varphi_{-N_i,\,s+N_i-u_i+t_j-\ep_j}\bigoplus_{i,j}\varphi_{-(N_i+N_j),\, s+u_i+N_j-u_j}\oplus\varphi_{-(N_i-N_j),\, s+u_i-u_j},
\end{align*}
where $\ep_i+\ep_j$ is viewed modulo 2 as before, and $\gamma_{ij}=2$ if $\ep_i=\ep_j=2$ and $0$ otherwise.
Hence
\begin{align*}
&L(\varphi\otimes\varphi^{\vee}\otimes\lambda_{0,s})\\
&=\prod_{i,j}L(\lambda_{\ep_i+\ep_j,\, s+t_i-t_j+2\ep_j-\gamma_{ij}})\prod_{i,j}L(\varphi_{-N_i,\, s+u_i-t_j+\ep_j} )\\
&\quad\prod_{i,j}L(\varphi_{-N_i,\, s+N_i-u_i+t_j-\ep_j})\prod_{i,j}L(\varphi_{-(N_i+N_j),\, s+u_i+N_j-u_j})L(\varphi_{-(N_i-N_j),\, s+u_i-u_j})\\
&=F(s)\prod_{i,j}\Gamma(\frac{s+t_i-t_j+2\ep_j-\gamma_{ij}}{2})\prod_{i,j}\Gamma(s+u_i-t_j+\ep_j)\\
&\qquad\qquad\prod_{i,j}\Gamma(s+N_i-u_i+t_j-\ep_j)\prod_{i,j}\Gamma(s+u_i+N_j-u_j)\\
&\qquad\qquad\prod_{i\geq j}\Gamma(s+u_i-u_j)\prod_{i<j}\Gamma(s+u_i-u_j-(N_i-N_j))\\
&=F(s)\prod_{i,j}\Gamma(\frac{s+t_i-t_j+2\ep_j-\gamma_{ij}}{2})\prod_{i,j}\Gamma(s+u_i-t_j+\ep_j)\\
&\qquad\qquad\prod_{i,j}\Gamma(s+N_i-u_i+t_j-\ep_j)\prod_{i,j}\Gamma(s+u_i+N_j-u_j)\\
&\qquad\qquad\prod_i\Gamma(s)\prod_{i< j}\Gamma(s+u_j-u_i)\Gamma(s+u_i-u_j-(N_i-N_j)),
\end{align*}
where $F(s)$ is a holomorphic function without a zero. We want this to be holomorphic at $s=1$.

To derive (a), assume $\Gamma(\frac{s+t_i-t_j+2\ep_j-\gamma_{ij}}{2})$ is holomorphic at $s=1$. If $t_i-t_j\notin\Z$, this is automatic. Assume $t_i-t_j\in\Z$. Then we must have either $t_i-t_j+2\ep_j-\gamma_{ij}\geq 0$ or $t_i-t_j+2\ep_j-\gamma_{ij}\in 2\Z$. The second condition is equivalent to $t_i-t_j\in2\Z$. For the first condition, by switching the roles of $i$ and $j$ we also have $t_j-t_i+2\ep_i-\gamma_{ij}\geq 0$. By combining the two, we obtain
\[
-2\ep_j+\gamma_{ij}\leq t_i-t_j\leq 2\ep_i-\gamma_{ij}.
\]
If $\ep_i=\ep_j$ then $0\leq t_i-t_j\leq 0$, which implies $t_i-t_j=0\in 2\Z$. If $\ep_i=0$ and $\ep_j=1$, then we have $-2\leq t_i-t_j\leq 0$. Hence either $t_i-t_j\in\{-2, 0\}\subseteq 2\Z$ or $t_j-t_i=1$. But the latter would give us a constituent of the form $\lambda_{0, t_i}\oplus\lambda_{1, t_i+1}$, which is considered as $\varphi_{0, t_i}$.

To derive (b), assume $\Gamma(s+u_i-t_j+\ep_j)$ is holomorphic at $s=1$. Then we must have $u_i-t_j+\ep_j\notin\Z^{<0}$. If $u_i-t_j\notin\Z$, this is automatic. If $u_i-t_j\in\Z$, then we must have $u_i-t_j+\ep_j\geq 0$, which implies $-\ep_j\leq u_i-t_j$. The other inequality of (3) follows in the same way from $\Gamma(s+N_i-u_i+t_j-\ep_j)$.

To derive (c), we argue in the same way by looking at $\Gamma(s+u_j-u_i)$ and $\Gamma(s+u_i-u_j-(N_i-N_j))$ for $i\leq j$.

As for the gamma function $\Gamma(s+u_i+N_j-u_j)$, if this is holomorphic at $s=1$ and $u_i-u_j\in\Z$, then we must have $u_i+N_j-u_j\geq 0$. By switching the roles of $i$ and $j$, we also have $u_j+N_i-u_i\geq 0$. By combining the two, we obtain $-N_j\leq u_i-u_j\leq N_i$. But this is subsumed under (3).

Hence we have proven that if $\pi$ is generic then the conditions (a), (b) and (c) are satisfied. The converse is clear by looking at $s=1$ in the above gamma functions.

\end{proof}

Let us note that the condition (3-c) in the above lemma is essentially the same as the complex case.

\subsection{Local Converse Theorem for $\GL_n(\R)$}
As we did in the complex case, we define
\[
\lambda_{\ep, t}\sim\lambda_{\ep', t'}\quad\text{if $t-t'\in\Z$},
\]
and
\[
\varphi_{-N, u}\sim\varphi_{-N', u'}\quad\text{if $u-u'\in\Z$}.
\]
Further we define
\[
\lambda_{\ep, t}\sim\varphi_{-N, u}\quad\text{if $t-u\in\Z$}.
\]
The relation $\sim$ is certainly an equivalence relation.


Let us first prove the following lemma.
\begin{Lem}
Let
\begin{align*}
\varphi&=\left(\la_{\ep_1,t_1}\oplus\cdots\oplus\la_{\ep_p,t_p}\right)\oplus
\left(\varphi_{-N_1,u_1}\oplus\cdots\oplus\varphi_{-N_q, u_q}\right)\quad\text{and}\\
\varphi'&=\left(\la_{\ep'_1,t'_1}\oplus\cdots\oplus\la_{\ep'_{p'},t'_{p'}}\right)\oplus
\left(\varphi_{-N'_1,u'_1}\oplus\cdots\oplus\varphi_{-N'_{q'} u'_{q'}}\right)
\end{align*}
be generic parameters such that all the constituents are equivalent under $\sim$. Further we assume
\[
t_1\preceq\cdots\preceq t_p\quad\text{and}\quad 0\leq N_1\leq\cdots\leq N_q,
\]
so the genericity condition (3-c) implies
\[
u_1\preceq\cdots\preceq u_p,
\]
and similarly for the $t_i'$'s, $u_i'$'s and $N_i'$'s.

Assume
\[
F(s)\gamma(s, \varphi, \psi_{\R})=\gamma(s, \varphi', \psi_{\R}),
\]
where $F(s)$ is a meromorphic function whose zeros and poles do not interfere with those of $\gamma(s, \varphi, \psi_{\R})$ and $\gamma(s, \varphi', \psi_{\R})$. Then $p=p'$ and $q=q'$, and $t_i-t_j'\in 2\Z$ for all $i, j$.
\end{Lem}
\begin{proof}
By computing (the reciprocals of) the gamma factors, we have
\begin{align*}
F(s)&\prod_{i=1}^p\frac{\Gamma\left(\frac{s+t_i}{2}\right)}{\Gamma\left(\frac{1-s-t_i+2\ep_i}{2}\right)}\cdot\prod_{i=1}^q\frac{\Gamma(s+u_i)}{\Gamma(1-s-u_i+N_i)}\\
=&\prod_{i=1}^{p'}\frac{\Gamma\left(\frac{s+t_i'}{2}\right)}{\Gamma\left(\frac{1-s-t_i'+2\ep_i'}{2}\right)}\cdot\prod_{i=1}^{q'}\frac{\Gamma(s+u_i')}{\Gamma(1-s-u_i'+N_i')},
\end{align*}
where $F(s)$ is a meromorphic function (possibly different from the one in the lemma) whose zeros and poles do not interfere with those of the gamma functions appearing here.

Since all the constituents are equivalent under $\sim$, we know $t_i-t_j\in\Z$, $u_i-u_j\in\Z$ and $t_i-u_j\in\Z$, and similarly for $t_i'$'s and $u_i'$'s. By the genericity condition (3-a), we know that $t_i-t_j\in 2\Z$ and $t_i'-t_j'\in 2\Z$. Hence either $t_i-t_j'\in 2\Z$ for all $i, j$ or $t_i-t_j'\in 2\Z+1$ for all $i, j$. Assume $t_i-t_j'\notin 2\Z$. Then $\prod_{i=1}^p\Gamma\left(\frac{s+t_i}{2}\right)$ and $\prod_{i=1}^p\Gamma\left(\frac{s+t_i'}{2}\right)$ do not share a pole. Hence by choosing $M\in\C$ comparable to $t_i$'s and ``large enough" with respect $\preceq$, we get a pole of order $p+q$ at $s=-M$ on the left-hand side, and a pole of order $q'$ on the right-hand side. (Note that by taking $M$ large enough, the denominators never have a pole at $s=-M$.) Hence we must have $p+q=q'$. By switching the roles of $\varphi$ and $\varphi'$ we have $p'+q'=q$. These two imply $p+p'=0$ (namely $p=p'=0$) and $q=q'$. Apparently in this case the assertion $t_i-t_j'\in 2\Z$ is vacuously true and $p+q=p'+q'$. If $p+p'\neq 0$ we must have $t_i-t_j'\in 2\Z$ for all $i, j$.

Next consider $M\in\C$ such that $-M+t_i+1\in 2\Z$ for some (and hence all) $i$. Then, at $s=-M$, none of $\Gamma\left(\frac{s+t_i}{2}\right)$ and $\Gamma\left(\frac{s+t_i'}{2}\right)$ has a pole. Further, by making $M$ ``large enough", we know that, at $s=-M$, all the denominators do not have a pole and the left-hand side has a pole of order $q$, namely the poles coming from $\Gamma(s+u_i)$. Similarly, we have a pole of order $q'$ on the right-hand side, which implies $q=q'$ and hence $p=p'$.
\end{proof}

Next we prove the following.
\begin{Prop}
Let
\begin{align*}
\varphi&=\left(\la_{\ep_1,t_1}\oplus\cdots\oplus\la_{\ep_p,t_p}\right)\oplus
\left(\varphi_{-N_1,u_1}\oplus\cdots\oplus\varphi_{-N_q, u_q}\right)\quad\text{and}\\
\varphi'&=\left(\la_{\ep'_1,t'_1}\oplus\cdots\oplus\la_{\ep'_{p'},t'_{p'}}\right)\oplus
\left(\varphi_{-N'_1,u'_1}\oplus\cdots\oplus\varphi_{-N'_{q'} u'_{q'}}\right)
\end{align*}
be generic parameters such that all the constituents are equivalent under $\sim$, where $t_i$'s, $N_i$'s, $t_i'$'s and $N_i'$'s are ordered as in the above lemma. Assume
\begin{equation}\label{E:equality_gamma_factor_R_lemma}
F_{\chi}(s)\gamma(s, \varphi\otimes\chi, \psi_{\R})=\gamma(s, \varphi'\otimes\chi, \psi_{\R})
\end{equation}
for all characters $\chi$, where $F_{\chi}(s)$ is a meromorphic function (depending on $\chi$) whose zeros and poles do not interfere with those of $\gamma(s, \varphi\otimes\chi, \psi_{\R})$ and $\gamma(s, \varphi'\otimes\chi, \psi_{\R})$. Then $\varphi=\varphi'$.
\end{Prop}
\begin{proof}
From the above lemma, we already know that $p=p'$, $q=q'$ and $t_i-t_j'\in 2\Z$.

By choosing $\chi$ to be trivial, (the reciprocal of) the identity \eqref{E:equality_gamma_factor_R_lemma} is written as
\begin{gather}\label{E:gamma_real_twist_by_lambda_0_s}
\begin{aligned}
&F(s)\prod_{i=1}^p\frac{\Gamma\left(\frac{s+t_i}{2}\right)}{\Gamma\left(\frac{1-s-t_i+2\ep_i}{2}\right)}\cdot\prod_{i=1}^q\frac{\Gamma(s+u_i)}{\Gamma(1-s-u_i+N_i)}\\
&\quad=\prod_{i=1}^{p}\frac{\Gamma\left(\frac{s+t_i'}{2}\right)}{\Gamma\left(\frac{1-s-t_i'+2\ep_i'}{2}\right)}\cdot\prod_{i=1}^{q}\frac{\Gamma(s+u_i')}{\Gamma(1-s-u_i'+N_i')},
\end{aligned}
\end{gather}
where $F(s)$ is a meromorphic function whose zeros and poles do not interfere with those of the gamma functions appearing here.

We will show $t_i=t_i'$ for $i=1,\dots, q$ by looking at poles of $\Gamma(\frac{s+t_i}{2})$ and $\Gamma(\frac{s+t'_i}{2})$. By the genericity condition (3-a), we know that $t_i-t_j\in2\Z$ for all $i, j$. Recall $t_i$'s and $t_i'$'s are in the increasing order, so that $t_1\preceq t_i$ and $t_1'\preceq t_i'$ for all $i$. Assume $t_1\prec t_1'$ (strict inequality). Let us consider the poles at $s=-t_1$. Since $\Gamma(\frac{s+t_1}{2})$ has a pole at $s=-t_1$, the numerator of the left-hand side has a pole. Certainly, the denominator $\Gamma(\frac{1-s-t_i+2\ep_i}{2})$ does not have a pole at $s=-t_1$ because $1+t_1-t_i+2\ep_i$ is odd. Also $\Gamma(1-s-u_i+N_i)$ does not have a pole because by the genericity condition (3-b) we have
\[
1-(-t_1)-u_i+N_i\geq 1+\ep_1\geq 1.
\]
Hence the left-hand side of \eqref{E:gamma_real_twist_by_lambda_0_s} has a pole. Now since we already know $t_1-t_1'\in 2\Z$, our assumption $t_1\prec t_1'$ actually implies $t_1\prec t_1'-1$. Apparently, on the right-hand side, $\Gamma\left(\frac{s+t_i'}{2}\right)$ cannot have a pole at $s=-t_1$ for all $i$. Hence some $\Gamma(s+u_i')$ must have a pole at $s=-t_1$. But the genericity condition implies
\[
-t_1+u_i'>-t_1'+1+u_i'\geq 1-\ep_1' \geq 0.
\]
Hence $\Gamma(s+u_i')$ cannot have a pole at $s=-t_1$, which is a contradiction. Thus we must have $t_1\succeq t_1'$. By switching the roles of $t_1$ and $t_1'$, we have $t_1\preceq t_1'$. Hence $t_1=t_1'$. Then we can cancel the gamma functions containing $t_1$ and $t_1'$. By repeating the same argument, we obtain
\[
t_i=t_i'
\]
for all $i=1,\dots, p$.

Next we show that all the $\ep$'s agree. For this purpose, we consider the twist by $\chi=\la_{1, s}$. By using \eqref{E:twisted_gamma_factor_real}, (the reciprocal of) the equality \eqref{E:equality_gamma_factor_R_lemma} is written as
\begin{align*}
&F(s)\prod_{i=1}^p\frac{\Gamma\left(\frac{s+t_i-2\ep_i}{2}\right)}{\Gamma\left(\frac{3-s-t_i}{2}\right)}
\cdot\prod_{i=1}^q\frac{\Gamma(s+u_i-1)}{\Gamma(1-s-u_i+1+N_i)}\\
&\quad=\prod_{i=1}^{p}\frac{\Gamma\left(\frac{s+t_i'-2\ep_i'}{2}\right)}{\Gamma\left(\frac{3-s-t_i'}{2}\right)}
\cdot\prod_{i=1}^{q}\frac{\Gamma(s+u_i'-1)}{\Gamma(1-s-u_i'+1+N_i')},
\end{align*}
where $F(s)$ is a meromorphic function whose zeros and poles do not interfere with those of the gamma functions appearing here. Now let $k\in\{1,\dots,p\}$ be such that
\[
t_k-2\ep_k\preceq t_i-2\ep_i
\]
for all $i$, namely $t_k-2\ep_k$ is minimal with respect to $\preceq$. Similarly let $\ell$ be such that
\[
t_\ell'-2\ep_\ell'\preceq t_i'-2\ep_i'
\]
for all $i$. One can then apply the same argument as above with $s=-(t_k-2\ep_k)$ and conclude that
\[
t_k-2\ep_k=t_\ell'-2\ep_\ell'.
\]
By arguing inductively, we have
\[
\{t_1-2\ep_1,\cdots, t_p-2\ep_p\}=\{t_1-2\ep'_1,\cdots, t_p-2\ep'_p\}
\]
as multisets.

From this identity of multisets, we will derive the identity
\[
\{(\ep_1,t_1),\cdots,(\ep_p, t_p)\}=\{(\ep'_1,t_1),\cdots,(\ep'_p, t_p)\}
\]
of multisets. For this, it suffices to show $(\ep_i, t_i)=(\ep'_j, t_j)$ for some $i$ and $j$, because then we can argue inductively on the size of the multisets. Now, we know $t_1-2\ep_1=t_i-2\ep_i'$ for some $i$. But then we must have $t_i-t_1=2(\ep_i'-\ep_1)\geq 0$ because of our ordering of $t_i$'s. So we must have $\ep_i'\geq \ep_1$. Suppose $\ep_1=1$. Then we have $\ep_i'=1$. Hence the equality $t_1-2\ep_1=t_i-2\ep_i'$ implies $t_1=t_i$, and so $(\ep_1, t_1)=(\ep_i', t_i)$. Next suppose $\ep_1=0$. If $\ep_1'=0$ then we have $(\ep_1, t_1)=(\ep_1', t_1)$. If $\ep_1'=1$ then by switching the roles of $\ep_1$ and $\ep_1'$ we have $(\ep_1', t_1)=(\ep_j, t_j)$ for some $j$. Thus in any case we know that $(\ep_i, t_i)=(\ep'_j, t_j)$ for some $i$ and $j$.

Now, we can cancel from \eqref{E:gamma_real_twist_by_lambda_0_s} all the factors containing $t_i, \ep_i, t_i'$ and $\ep_i'$ and obtain
\begin{equation}\label{E:real_2-dimensional_part}
F(s)\prod_{i=1}^q\frac{\Gamma(s+u_i)}{\Gamma(1-s-u_i+N_i)}
=\prod_{i=1}^{q}\frac{\Gamma(s+u_i')}{\Gamma(1-s-u_i'+N_i')}.
\end{equation}

Recall that $u_i$'s and $u_i'$'s are in the increasing order with respect to $\preceq$. We will show $u_i=u_i'$ by induction on $i$. Assume $u_1\prec u_1'$ (strict inequality). Then $\Gamma(s+u_1)$ has a pole at $s=-u_1$ on the left-hand side, and the denominator $\Gamma(1-s-u_i+N_i)$ does not have a pole at $s=-u_1$ because by the genericity condition (3-c) we have $1+u_1-u_i+N_i\geq 1+N_1$ for all $i$. But since $u_1\prec u_1'$, the right-hand side cannot have a pole at $s=-u_1$. Hence we must have $u_1\succeq u_1'$. By switching the roles of $u_1$ and $u_1'$, we have $u_1\preceq u_1'$, from which we have $u_1=u_1'$. Now assume we have shown $u_i=u_i'$ for $i=1,\dots, j$ for some $j$. Then the above identity \eqref{E:real_2-dimensional_part} is reduced to
\begin{gather}\label{E:induction_step_real_2-dimensional_part}
\begin{aligned}
F(s)&\prod_{i=1}^j\frac{1}{\Gamma(1-s-u_i+N_i)}\prod_{i=j+1}^q\frac{\Gamma(s+u_i)}{\Gamma(1-s-u_i+N_i)}\\
=&\prod_{i=1}^j\frac{1}{\Gamma(1-s-u_i'+N_i')}\prod_{i=j+1}^q\frac{\Gamma(s+u_i')}{\Gamma(1-s-u_i'+N_i')}.
\end{aligned}
\end{gather}
Assume $u_{j+1}\prec u_{j+1}'$. Then by the same reasoning as above, the product
\[
\prod_{i=j+1}^q\frac{\Gamma(s+u_i)}{\Gamma(1-s-u_i+N_i)}
\]
has a pole at $s=-u_{j+1}$. Also $\Gamma(1-s-u_i+N_i)$ does not have a pole at $s=-u_{j+1}$ for all $i=1,\dots, j$, because $1+u_{j+1}-u_i+N_i\geq 1+N_i\geq 1$ by genericity condition (3-c), since $j+1 \geq i$. Hence the left-hand side of \eqref{E:induction_step_real_2-dimensional_part} has a pole at $s=-u_{j+1}$. But the right-hand side does not have a pole at $s=u_{j+1}$ because $u_{j+1}\prec u_{j+1}'$. Thus we must have $u_{j+1}\succeq u_{j+1}'$. By switching the roles of $u_{j+1}$ and $u_{j+1}'$, we get $u_{j+1}\preceq u_{j+1}'$, from which we have $u_{j+1}=u_{j+1}'$. Hence we have
\[
u_i=u_i'
\]
for all $i=1,\dots, q$.

By cancelling the numerators from \eqref{E:real_2-dimensional_part}, we obtain
\[
\prod_{i=1}^q\Gamma(1-s-u_i+N_i)=F(s)\prod_{i=1}^{q}\Gamma(1-s-u_i'+N_i').
\]
Let $k, \ell\in\{1,\dots,q\}$ be such that
\[
u_k-N_k\succeq u_i-N_i\quad\text{and}\quad u_\ell'-N_\ell'\succeq u_i'-N_i'
\]
for all $i$. Then at $s=-u_k+N_k$ the left-hand side (and hence the right-hand side) is holomorphic, which implies $u_k-N_k\preceq u_\ell'-N_\ell'$. By switching the roles we obtain $u_k-N_k\succeq u_\ell'-N_\ell'$ and hence $u_k-N_k= u_\ell'-N_\ell'$. By arguing inductively we obtain
\[
\{u_1-N_1,\dots, u_q-N_q\}=\{u_1'-N_1',\dots,u_q'-N_q'\}
\]
as multisets.

Then we can show $N_i=N_i'$ for all $i=1,\dots, n$ by exactly the same argument as the complex case as follows. By the above identity of the multisets we must have $u_1-N_1=u_i'-N_i'$ for some $i$. Since we already know $u_i=u_i'$, we have
\[
N_i'-N_1=u_i'-u_1=u_i'-u_1'\leq N_i'-N_1',
\]
where the last inequality is by the genericity condition (3). Hence we have $N_1'\leq N_1$. Also we must have $u_1'-N_1'=u_j-N_j$ for some $j$. By applying the same argument, we must have $N_1\leq N_1'$. Thus we must have $N_1=N_1'$. By arguing inductively we have
\[
N_i=N_i'
\]
for all $i=1,\dots, n$.


\end{proof}

Now, we are ready to prove the local converse theorem.
\begin{Thm}
Let $\pi$ and $\pi'$ be generic irreducible admissible representations of $\GL_n(\R)$. Assume that
\[
\gamma(s, \pi\times\chi, \psi_{\R})=\gamma(s, \pi'\times\chi, \psi_{\R})
\]
for all unitary characters $\chi$. Then $\pi=\pi'$.
\end{Thm}
\begin{proof}
As in the complex case, we may assume that the identity of the gamma factors holds for all (not necessarily unitary) characters $\chi$.

Let $\varphi$ and $\varphi'$ be the Langlands parameters of $\GL_n(\R)$ corresponding to $\pi$ and $\pi'$, respectively. Let us write
\[
\varphi=\varphi_1\oplus\cdots\oplus\varphi_k\quad\text{and}\quad\varphi'=\varphi_1'\oplus\cdots\oplus\varphi'_{k'},
\]
where all the constituents of each $\varphi_j$ are equivalent under $\sim$ and the constituents of $\varphi_i$ and $\varphi_j$ are inequivalent under $\sim$  for $i\neq j$, and similarly for $\varphi'$. We then have
\[
\prod_{i=1}^k\gamma(s, \varphi_i\otimes\chi, \psi_{\R})=\prod_{i=1}^{k'}\gamma(s, \varphi_i'\otimes\chi, \psi_{\R}).
\]
Note that, for $i\neq j$, the gamma factors $\gamma(s, \varphi_i\otimes\chi, \psi_{\R})$ and $\gamma(s, \varphi_j\otimes\chi, \psi_{\R})$ do not share a zero or a pole, and similarly for $\varphi'$.

Assume that $\varphi$ and $\varphi'$ do not share any constituents equivalent under $\sim$. Then $\gamma(s, \varphi\otimes\chi, \psi_{\R})$ and $\gamma(s, \varphi'\otimes\chi, \psi_{\R})$ do not share a zero or pole. So there are at least some $\varphi_i$ and $\varphi_j'$ having constituents equivalent under $\sim$. By reordering indices, we may assume $i=j=1$. Then the equality of the gamma factors is written as
\[
F_{\chi}(s)\gamma(s,\varphi_1\otimes\chi, \psi_{\R})=\gamma(s, \varphi_1'\otimes\chi, \psi_{\R}),
\]
where $F_{\chi}(s)$ is a meromorphic function whose poles and zeros do not interfere with those of the above two gamma factors. Hence by the above proposition we have $\varphi_1=\varphi_1'$. Arguing inductively, we conclude $\varphi=\varphi'$.
\end{proof}

\bibliographystyle{amsalpha}
\bibliography{arch_local_converse}

\end{document}